\documentclass [12pt]{amsart}
\usepackage{amsmath,amssymb}
 \usepackage{amsthm, amsfonts}
 \usepackage{enumerate}

\newtheorem{theorem}{Theorem}[section]

\newtheorem{lemma}[theorem]{Lemma}
\newtheorem{corollary}[theorem]{Corollary}
\theoremstyle{definition}
\newtheorem{definition}[theorem]{Definition}
\newtheorem{example}[theorem]{Example}

\theoremstyle{remark}

\numberwithin{equation}{section}

\begin{document}

\title [{{On Graded classical 2-absorbing second submodules }}]{On Graded classical 2-absorbing second submodules of graded modules over graded commutative rings}

 \author[{{K. Al-Zoubi and Mariam  Al-Azaizeh }}]{\textit{Khaldoun Al-Zoubi* and  Mariam  Al-Azaizeh}}

\address
{\textit{Khaldoun Al-Zoubi, Department of Mathematics and
Statistics, Jordan University of Science and Technology, P.O.Box
3030, Irbid 22110, Jordan.}}
\bigskip
{\email{\textit{kfzoubi@just.edu.jo}}}
\address
{\textit{ Mariam  Al-Azaizeh, Department of Mathematics, University of Jordan, Amman, Jordan..}}
\bigskip
{\email{\textit{ maalazaizeh15@sci.just.edu.jo}}}


\date{}
\begin{abstract}
 Let $G$ be a group with identity $e$. Let $R$ be a $G$-graded commutative
ring and $M$ a graded $R$-module. In this paper, we introduce the concept of graded classical and graded
strongly classical 2-absorbing second submodules of graded modules over a
graded commutative rings. A number of results concerning of these classes
of graded submodules and their homogeneous components are given.
\end{abstract}

\keywords{graded classical 2-absorbing second submodules, graded strongly classical 2-absorbing second submodules, graded 2-absorbing second submodules. \\
$*$ Corresponding author}
 \maketitle


 \section{Introduction}
  Badawi in \cite{19} introduced the concept of 2-absorbing ideals of
commutative rings.
Later on, Anderson and Badawi in \cite{13} generalized the concept of
2-absorbing ideals of commutative rings to the concept of $n$-absorbing
ideals of commutative rings for every positive integer $n\geq 2.$
In light of \cite{19} and \cite{13}, many authors studied the
concept of 2-absorbing submodules and $n$-absorbing submodules.
In \cite{17}, the authors introduced and studied the
concepts of 2-absorbing and strongly 2-absorbing second submodules.
In \cite{23}, the authors introduced and studied the concept of
classical 2-absorbing submodules as a generalization of
2-absorbing submodules. Recently, H. Ansari-Toroghy and F. Farshadifar in \cite{14} introduced
and studied the concepts of classical and strongly classical 2-absorbing
second submodules of modules over commutative rings.

The scope of this paper is devoted to the theory of graded modules over
graded commutative rings. One use of rings and modules with gradings is in
describing certain topics in algebraic geometry. Here, in particular, we are
dealing with graded classical and graded strongly classical 2-absorbing
second submodules of graded modules over graded commutative rings.

Graded prime ideals of graded commutative rings have been introduced and
studied in [10, 30-32].
S.E. Atani in \cite{18} extended graded prime ideals to graded prime
submodules. Several authors investigated properties of graded prime
submodules, for examples see \cite{ 1, 9, 11, 29}.
The concept of graded 2-absorbing ideals, generalizations of graded
prime ideals, were studied by K. Al-Zoubi and R. Abu-Dawwas, and other
authors, (see \cite{3, 25}.)
K. Al-Zoubi and R. Abu-Dawwas in \cite{2} extended graded
2-absorbing ideals to graded 2-absorbing submodules.
Later on, M. Hamoda and A. E. Ashour in \cite{24} introduced the
concept of graded n-absorbing submodules that is a generalization of the
concept of graded prime ideals.
The notion of graded classical prime submodules as a generalization
of graded prime submodules was introduced in \cite{21} and studied in%
\cite{ 9, 6, 8, 12}.
The notion of graded second submodules was introduced in \cite{15} and
studied in \cite{7, 16, 20}.
In\cite{4}, the authors introduced and studied the concept of
graded classical 2-absorbing submodules as a generalization of
graded 2-absorbing submodules.

Recently, K. Al-Zoubi and M. Al-Azaizeh in \cite{5} introduced and
studied the concepts of graded 2-absorbing and graded strongly 2-absorbing
second submodules.

Here, we introduce the concept of graded classical (resp. graded strongly
classical) 2-absorbing second submodules of graded modules over a
commutative graded rings as a generalization of graded 2-absorbing (resp.
graded strongly 2-absorbing) second submodules and investigate some
properties of these classes of graded submodules.

 \section{Preliminaries}

\textbf{Convention}. Throughout this paper all rings are commutative with identity and all modules are unitary.
First, we recall some basic properties of graded rings and modules which
will be used in the sequel. We refer to \cite{22}, \cite{26}, \cite{27}
and \cite{28} for these basic properties and more information on graded
rings and modules. Let $G$ be a multiplicative group and $e$ denote the identity element of $G$. A ring $R$ is called a graded ring (or $G$-graded ring) if there exist additive subgroups $R_{\alpha }$ of $R$ indexed by the elements $\alpha \in G
$ such that $R=\bigoplus_{\alpha \in G}R_{\alpha }$ and $R_{\alpha }R_{\beta
}\subseteq R_{\alpha \beta }$ for all $\alpha $, $\beta \in G$.
The elements of $R_{\alpha }$ are called homogeneous of degree $\alpha $ and
all the homogeneous elements are denoted by $h(R)$, i.e. $h(R)=\cup _{\alpha
\in G}R_{\alpha }$. If $r\in R$, then $r$ can be written uniquely as $%
\sum_{\alpha \in G}r_{\alpha }$, where $r_{\alpha }$ is called a homogeneous
component of $r$ in $R_{\alpha }$. Moreover, $R_{e}$ is a subring of $R$ and
$1\in R_{e}$. Let $R=\bigoplus_{\alpha \in G}R_{\alpha }$ be a $G$-graded
ring. An ideal $I$ of $R$ is said to be a graded ideal if $%
I=\bigoplus_{\alpha \in G}(I\cap R_{\alpha }):=\bigoplus_{\alpha \in
G}I_{\alpha }$.
Let $R=\bigoplus_{\alpha \in G}R_{\alpha }$ be a $G$-graded ring. A Left $R$%
-module $M$ is said to be \textit{a graded }$R$\textit{-module} (or $G$%
\textit{-graded }$R$\textit{-module}) if there exists a family of additive
subgroups $\{M_{\alpha }\}_{\alpha \in G}$ of $M$ such that $%
M=\bigoplus_{\alpha \in G}M_{\alpha }$ and $R_{\alpha }M_{\beta }\subseteq
M_{\alpha \beta }$ for all $\alpha ,\beta \in G$. Here, $R_{\alpha }M_{\beta
}$ denotes the additive subgroup of $M$ consisting of all finite sums of
elements $r_{\alpha }m_{\beta }$ with $r_{\alpha }\in R_{\alpha }$ and $%
m_{\beta }\in M_{\beta }.$ Also if an element of $M$ belongs to $\cup
_{\alpha \in G}M_{\alpha }=h(M)$, then it is called a homogeneous. Note
that $M_{\alpha }$ is an $R_{e}$-module for every $\alpha \in G$. So, if $%
I=\bigoplus_{\alpha \in G}I_{\alpha }$ is a graded ideal of $R$, then $%
I_{\alpha }$ is an $R_{e}$-module for every $\alpha \in G$. Let $%
R=\bigoplus_{\alpha \in G}R_{\alpha }$ be a $G$-graded ring. A submodule $N
$ of $M$ is said to be \textit{a graded submodule of }$M$ if $%
N=\bigoplus_{\alpha \in G}(N\cap M_{\alpha }):=\bigoplus_{\alpha \in
G}N_{\alpha }.$ In this case, $N_{\alpha }$ is called the $\alpha $%
-component of $N$.

Let $R$ be a $G$-graded ring and $M$ a graded $R$-module.
A proper graded submodule $C$ of $M$ is said to be \textit{a completely
graded irreducible} if $C=\cap _{\alpha \in \Delta }C_{\alpha },$ where $%
\{C_{\alpha }\}_{\alpha \in \Delta }$ is a family of graded submodules of $M$%
, implies that $C=C_{\beta }$ for some $\beta \in \Delta $ (see \cite{5}.) A non-zero graded submodule $S$ of $M$ is said to be \textit{a graded
2-absorbing second submodule of }$M$ if whenever $r,$ $t\in h(R),$
$C$ is a completely graded irreducible submodule of $M$, and $rtS\subseteq C,
$ then $rS\subseteq C$ or $tS\subseteq C$ or $rt\in Ann_{R}(S)$ (see \cite{5}.) A non-zero graded submodule $S$ of $M$ is said to be \textit{a graded
strongly 2-absorbing second submodule of} $M$ if whenever $r,$ $t\in h(R),$ $%
C_{1},C_{2}$ are completely graded irreducible submodules of $M$, and $%
rtS\subseteq C_{1}\cap C_{2},$ then $rS\subseteq C_{1}\cap C_{2}$ or $%
tS\subseteq C_{1}\cap C_{2}$ or $rt\in Ann_{R}(S)$ (see \cite{5}.) A
proper graded submodule $C$ of $M$ is said to be a graded classical
2-absorbing submodule of $M$ if $C\neq M$; and whenever $r,s,t\in h(R)$ and $%
m\in h(M)$ with $rstm\in C$, then either $rsm\in C$ or $rtm\in C$ or $stm\in C
$ (see \cite{4}.)



 \section{Graded classical 2-absorbing second submodules}
\begin{definition}
Let $R$ be a $G$-graded ring and $M$ a graded $R$-module. A non-zero graded submodule $C$ of $M$ is said to be \textit{a graded
classical 2-absorbing second submodule of $M,$} if whenever $%
r_{\alpha },$ $s_{\beta }$, $t_{\gamma }\in h(R)$, $U$  is a completely
graded irreducible submodule of $M$ and $r_{\alpha }s_{\beta }t_{\gamma
}C\subseteq U$, then either $r_{\alpha }s_{\beta }C\subseteq U$ or $s_{\beta
}t_{\gamma }C\subseteq U$ or $r_{\alpha }t_{\gamma }C\subseteq U$. We say
that $M$ is \textit{ a graded classical 2-absorbing second module }if $M$ is a graded classical 2-absorbing second submodule of itself.
\end{definition}

\begin{lemma}
Let $R$ be a $G$-graded ring, $M$ a graded $R$-module and $C$ a graded classical 2-absorbing second submodule of $M$.
Let $I=\bigoplus_{\gamma \in G}I_{\gamma }$ be a graded ideal of $R$. Then
for every $r_{\alpha },s_{\beta }$ $\in h(R)$, $\gamma \in G$ and completely
graded irreducible submodule $U$ of $M$ with $r_{\alpha }s_{\beta
}I_{\gamma }C\subseteq U$, either $r_{\alpha }I_{\gamma }C\subseteq U$ or $\
s_{\beta }I_{\gamma }C\subseteq U$ or $r_{\alpha }s_{\beta }C\subseteq U$.
\end{lemma}
\begin{proof}
Let $r_{\alpha },s_{\beta }$ $\in h(R)$, $\gamma \in G$
and $U$ be a completely graded irreducible submodule of $M$ such that $%
r_{\alpha }s_{\beta }I_{\gamma }C\subseteq U$, $r_{\alpha }I_{\gamma
}C\nsubseteq U$ and $s_{\beta }I_{\gamma }C\nsubseteq U.$ We have to show
that $r_{\alpha }s_{\beta }C\subseteq U.$ By our assumption there exist $%
i_{\gamma },$ $i_{\gamma }^{\prime }\in I_{\gamma }$ such that $r_{\alpha }\
i_{\gamma }C\nsubseteq U$ and $s_{\beta }\ i_{\gamma }^{\prime }C\nsubseteq
U.$ As $r_{\alpha }s_{\beta }i_{\gamma }C\subseteq U$ and $C$ is a graded
classical 2-absorbing second submodule, we have either $r_{\alpha }s_{\beta
}C\subseteq U$ or $\ s_{\beta }i_{\gamma }C\subseteq U.$
Similarly, by $r_{\alpha }s_{\beta }i_{\gamma }^{\prime }C\subseteq U,$ we
get either $r_{\alpha }s_{\beta }C\subseteq U$ or $r_{\alpha }i_{\gamma
}^{\prime }C\subseteq U.$ If $r_{\alpha }s_{\beta }C\subseteq U,$ we are
done. Suppose that $r_{\alpha }s_{\beta }C\nsubseteq U,$ which implies that $%
s_{\beta }i_{\gamma }C\subseteq U$ and $\ r_{\alpha }i_{\gamma }^{\prime
}C\subseteq U.\ $By $(i_{\gamma }+i_{\gamma }^{\prime })\in I_{\gamma },$ it
follows that $\ r_{\alpha }s_{\beta }(i_{\gamma }+i_{\gamma }^{\prime
})C\subseteq U.$ Then either $\ r_{\alpha }(i_{\gamma }+i_{\gamma }^{\prime
})C\subseteq U$ or $s_{\beta }(i_{\gamma }+i_{\gamma }^{\prime })C\subseteq U
$ as $C$ is a graded classical 2-absorbing second submodule of $M.$ If $%
r_{\alpha }(i_{\gamma }+i_{\gamma }^{\prime })C\subseteq U,$ then $r_{\alpha
}i_{\gamma }C\subseteq U$ since $r_{\alpha }i_{\gamma }^{\prime
}C\subseteq U,$ which is a contradiction. Similarly, if $s_{\beta
}(i_{\gamma }+i_{\gamma }^{\prime })C\subseteq U,$ we get a contradiction.
Thus $r_{\alpha }s_{\beta }C\subseteq U.$
\end{proof}

\begin{lemma}
Let $R$ be a $G$-graded ring, $M$ a graded $R$-module
and $C$ a graded classical 2-absorbing second submodule of $M$. Let $%
I=\bigoplus_{\beta \in G}I_{\beta }$ and $J=\bigoplus_{\gamma \in
G}J_{\gamma }$\ be a graded ideals of $R.$ Then for every $r_{\alpha }$ $\in
h(R)$, $\beta ,\gamma \in G$ and completely graded irreducible submodule $U$
of $M$ \ with $r_{\alpha }I_{\beta }J_{\gamma }C\subseteq U$, either $\
r_{\alpha }I_{\beta }C\subseteq U$ or $r_{\alpha }J_{\gamma }C\subseteq U$
or $\ I_{\beta }J_{\gamma }C\subseteq U.$
\end{lemma}
\begin{proof}
Let $r_{\alpha }$ $\in h(R)$, $\beta ,\gamma \in G$ and $U$
be a completely graded irreducible submodule of $M$ such that $r_{\alpha
}I_{\beta }J_{\gamma }C\subseteq U,$ $r_{\alpha }I_{\beta }C\nsubseteq U$
and $r_{\alpha }J_{\gamma }C\nsubseteq U.$ We have to show that $I_{\beta
}J_{\gamma }C\subseteq U.$ Let $i_{\beta }\in I_{\beta }$ and $j_{\gamma
}\in J_{\gamma }.$ By our\ assumption, there exist $i_{\beta }^{\prime }\in
I_{\beta }$ and $j_{\gamma }^{\prime }\in J_{\gamma }$ such that $r_{\alpha
}i_{\beta }^{\prime }C\nsubseteq U$ and $r_{\alpha }j_{\gamma }^{\prime
}C\nsubseteq U.$ Since $r_{\alpha }i_{\beta }^{\prime }J_{\gamma }C\subseteq
U,r_{\alpha }i_{\beta }^{\prime }C\nsubseteq U$ and $r_{\alpha }J_{\gamma
}C\nsubseteq U,$ by Lemma 3.2, we get $i_{\beta }^{\prime
}J_{\gamma }C\subseteq U.$ Similarly, by $\ r_{\alpha }j_{\gamma }^{\prime
}I_{\beta }C\subseteq U,$ $r_{\alpha }j_{\gamma }^{\prime }C\nsubseteq U$
and $r_{\alpha }I_{\beta }C\nsubseteq U,$ we get $_{\ }j_{\gamma }^{\prime
}I_{\beta }C\subseteq U.$ By $(i_{\beta }+i_{\beta }^{\prime })\in I_{\beta
}\ $and $(j_{\gamma }+j_{\gamma }^{\prime })\in J_{\gamma },$ it follow that
$r_{\alpha }(i_{\beta }+i_{\beta }^{\prime })(j_{\gamma }+j_{\gamma
}^{\prime })C\subseteq U.$ Since $C$ is a graded classical 2-absorbing
second submodule, we have either \ $r_{\alpha }(i_{\beta }+i_{\beta
}^{\prime })C\subseteq U$ or $r_{\alpha }\ (j_{\gamma }+j_{\gamma }^{\prime
})C\subseteq U$ or $\ (i_{\beta }+i_{\beta }^{\prime })(j_{\gamma
}+j_{\gamma }^{\prime })C\subseteq U.$ If $r_{\alpha }(i_{\beta }+i_{\beta
}^{\prime })C\subseteq U,$ then $r_{\alpha }i_{\beta }C\nsubseteq U$ since $%
r_{\alpha }i_{\beta }^{\prime }C\nsubseteq U.$ Since $r_{\alpha }i_{\beta
}J_{\gamma }C\subseteq U,$ $r_{\alpha }J_{\gamma }C\nsubseteq U$ and $%
r_{\alpha }i_{\beta }C\nsubseteq U,$ by Lemma 3.2, we get$\
i_{\beta }J_{\gamma }C\subseteq U$ and hence $i_{\beta }j_{\gamma
}C\subseteq U.$ Similarly, if $r_{\alpha }\ (j_{\gamma }+j_{\gamma }^{\prime
})C\subseteq U,$ we conclude that $i_{\beta }j_{\gamma }C\subseteq U.$ If $%
(i_{\beta }+i_{\beta }^{\prime })(j_{\gamma }+j_{\gamma }^{\prime
})C\subseteq U,$ then $i_{\beta }j_{\gamma }C+i_{\beta }j_{\gamma }^{\prime
}C+i_{\beta}^{\prime }j_{\gamma }C+i_{\beta }^{\prime }j_{\gamma }^{\prime
}C\subseteq U$, this implies that \textbf{\ }$i_{\beta }j_{\gamma
}C\subseteq U.$ Thus $I_{\beta }J_{\gamma }C\subseteq U.$
\end{proof}

The following theorem give us a characterization of graded classical 2-absorbing second submodule of a graded module.

\begin{theorem}
Let $R$ be a $G$-graded ring, $M$ a graded $R$-module
and $C$ a non-zero graded submodule of $M$. Let $I=\bigoplus_{\alpha \in G}I_{\alpha },$ $J=\bigoplus_{\beta \in
G}J_{\beta }$ and $K=\bigoplus_{\gamma \in G}K_{\gamma }$ be a graded ideals
of $R$. Then the following statements are equivalent:
\begin{enumerate}[\upshape (i)]

\item $C$ is a graded classical 2-absorbing second submodule of $M$;

\item For every $\alpha ,$ $\beta ,$ $\gamma \in G$ and completely graded
irreducible submodule $U$ of $M\ $ with $I_{\alpha }J_{\beta }K_{\gamma
}C\subseteq U$, either $J_{\beta }K_{\gamma }C\subseteq U$ or $I_{\alpha
}J_{\beta }C\subseteq U$ or $I_{\alpha }K_{\gamma }C\subseteq U.$
\end{enumerate}
\end{theorem}
\begin{proof}
$(i) \Rightarrow (ii)$Assume that $C$ is a graded classical 2-absorbing
second submodule of $M.$ Let $U$ be a completely graded irreducible
submodule of $M$ and $\alpha ,$ $\beta ,$ $\gamma \in G$ such that $%
I_{\alpha }J_{\beta }K_{\gamma }C\subseteq U$ and $J_{\beta }K_{\gamma
}C\nsubseteq U.$ Then by Lemma 3.3, for all $i_{\alpha }\in
I_{\alpha },$ we have either $i_{\alpha }J_{\beta }C\subseteq U$ or $%
i_{\alpha }K_{\gamma }C\subseteq U.$ If $i_{\alpha }J_{\beta }C\subseteq U$
\ for all $i_{\alpha }\in I_{\alpha },$ then $I_{\alpha }J_{\beta
}C\subseteq U,$ we are done. Similarly, if $_{\ }i_{\alpha }K_{\gamma
}C\subseteq U$ for all $i_{\alpha }\in I_{\alpha },$ then $I_{\alpha
}K_{\gamma }C\subseteq U,$ we are done. Suppose that there exist \textbf{\ }$%
i_{\alpha },$ $i_{\alpha }^{\prime }\in I_{\alpha }$ such that $i_{\alpha
}J_{\beta }C\nsubseteq U$ and $i_{\alpha }^{\prime }K_{\gamma }C\nsubseteq U.
$ Since $i_{\alpha }J_{\beta }K_{\gamma }C\subseteq U,$ $i_{\alpha }J_{\beta
}C\nsubseteq U$ and $J_{\beta }K_{\gamma }C\nsubseteq U,$ by Lemma
3.3, we get $i_{\alpha }\ K_{\gamma }C\subseteq U.$ Similarly, by $%
i_{\alpha }^{\prime }J_{\beta }K_{\gamma }C\subseteq U,$ $i_{\alpha
}^{\prime }K_{\gamma }C\nsubseteq U$ and $J_{\beta }K_{\gamma }C\nsubseteq U,
$ we get $i_{\alpha }^{\prime }J_{\beta }C\subseteq U.$ Since $(i_{\alpha
}+i_{\alpha }^{\prime })J_{\beta }K_{\gamma }C\subseteq U\ $and $J_{\beta
}K_{\gamma }C\nsubseteq U$, by Lemma 3.3, we get either $(i_{\alpha
}+i_{\alpha }^{\prime })J_{\beta }C\subseteq U$ or $(i_{\alpha }+i_{\alpha
}^{\prime })K_{\gamma }C\subseteq U.$ If $(i_{\alpha }+i_{\alpha }^{\prime
})J_{\beta }C\subseteq U,$ then $i_{\alpha }^{\prime }J_{\beta }C\nsubseteq U
$ since $i_{\alpha }J_{\beta }C\nsubseteq U,$ which is a contradiction.
Similarly, if $(i_{\alpha }+i_{\alpha }^{\prime })K_{\gamma }C\subseteq U,$
we get a contradiction. Therefore either $I_{\alpha }J_{\beta }C\subseteq U$
or $I_{\alpha }K_{\gamma }C\subseteq U.$

 $(ii) \Rightarrow (i)$ Assume that $(ii)$ holds. Let $r_{\alpha },$ $%
s_{\beta },$ $t_{\gamma }\in h(R)$ and $U$ be a completely graded
irreducible submodule of $M$ such that $r_{\alpha }s_{\beta }t_{\gamma
}C\subseteq U.$ \ Let $I=r_{\alpha }R,$ $J=s_{\beta }R\ $and $K=t_{\gamma }R$
be a graded ideals of $R$ generated by $r_{\alpha }$, $s_{\beta }$ and $%
t_{\gamma },$ respectively. Then $I_{\alpha }J_{\beta }K_{\gamma }C\subseteq
U.$ By our assumption, we have either $J_{\beta }K_{\gamma }C\subseteq U$ or
$I_{\alpha }J_{\beta }C\subseteq U$ or $I_{\alpha }K_{\gamma }C\subseteq U.$
This yields that either $s_{\beta }t_{\gamma }C\subseteq U$ or $r_{\alpha
}s_{\beta }C\subseteq U$ or $r_{\alpha }t_{\gamma }C\subseteq U.$ Thus $C$
is a graded classical 2-absorbing second submodule of $M$.
\end{proof}
\begin{corollary}
Let $R$ be a $G$-graded ring, $M$ a graded $R$-module and $C$ a graded classical 2-absorbing
second submodule of $M$ and $I=\bigoplus_{\alpha \in G}I_{\alpha }$ be a
graded ideal of $R.$ Then for each $\alpha \in G,$ $I_{\alpha
}^{n}C=I_{\alpha }^{n+1}C$ for all $n\geq 2.$
\end{corollary}
\begin{proof}
It is enough to show that $I_{\alpha }^{2}C=I_{\alpha }^{3}C.$ It is clear
that $I_{\alpha }^{3}C\subseteq I_{\alpha }^{2}C.$ Let $U$ be a completely
graded irreducible submodule of $M$ such that $I_{\alpha }^{3}C\subseteq U.$ By Theorem 3.4, we get $I_{\alpha }^{2}C\subseteq U$. This yields
that $\ I_{\alpha }^{2}C\subseteq I_{\alpha }^{3}C$ by \cite[Lemma 2.3]{5}. Therefore $I_{\alpha }^{2}C=I_{\alpha }^{3}C.$
\end{proof}
Clearly every graded 2-absorbing second submodule is a graded classical
2-absorbing second submodule. The following example shows that the converse
is not true in general.
\begin{example}
Let $G=%
\mathbb{Z}
_{2},$ then $R=%
\mathbb{Z}
$ is a $G$-graded ring with $R_{0}=%
\mathbb{Z}
$ and $R_{1}=\{0\}.$ Let $M=%
\mathbb{Z}
\ $ be a graded $R$-module with $M_{0}=%
\mathbb{Z}
\ $ and $M_{1}=\{0\}.$ Now, consider the graded submodule $N=2%
\mathbb{Z}
.$ Then $N$ is a graded classical 2-absorbing second submodule of $R$ which is not a graded 2-absorbing second submodule
of $M.$
\end{example}


Let $R$ be a $G$-graded ring, $M$ a graded $R$-module and $N$ a graded
submodule of $M$. Then $(N:_{R}M)$ is defined as $(N:_{R}M)=\{r\in
R|rM\subseteq N\}.$ It is shown in \cite[Lemma 2.1]{18} that if $N$ is a
graded submodule of $M$, then $\ (N:_{R}M)=\{r\in R:rM\subseteq N\}$ is a
graded ideal of $R$. The annihilator of $M$ is defined as $(0:_{R}M)$ and is
denoted by $Ann_{R}(M).$
Recall that a proper graded ideal $I$ of a G-graded ring $R$ is said to be
\textit{a graded 2-absorbing ideal of }$R$ if whenever $r,s,t\in h(R)$ with $%
rst\in I$, then $rs\in I$ or $rt\in I$ or $st\in I$ (see \cite{3}.)

\begin{theorem}
Let $R$ be a $G$-graded ring, $M$ a graded $R$-module
and $C$ a graded classical 2-absorbing second submodule of $M$. Then for
each completely graded irreducible submodule $U$ of $M$ with $C\nsubseteq U,$
$(U:_{R}C)$ is a graded 2-absorbing ideal of $R.$
\end{theorem}
\begin{proof}
Let $r_{\alpha },$ $s_{\beta },$ $t_{\gamma }\in h(R)$ such
that $r_{\alpha }s_{\beta }t_{\gamma }\in (U:_{R}C).$ Then $r_{\alpha
}s_{\beta }t_{\gamma }C\subseteq U.$ Since $C$ is a graded classical
2-absorbing second submodule, we have either $s_{\beta }t_{\gamma
}C\subseteq U$ or $r_{\alpha }s_{\beta }C\subseteq U$ or $r_{\alpha
}t_{\gamma }C\subseteq U,$ it follows that either $\ s_{\beta }t_{\gamma
}\in (U:_{R}C)$ or $r_{\alpha }s_{\beta }\in (U:_{R}C)$ or $r_{\alpha
}t_{\gamma }\in (U:_{R}C).$ Thus $(U:_{R}C)$ is a graded 2-absorbing ideal
of $R.$
\end{proof}
\begin{theorem}
Let $R$ be a $G$-graded ring, $M$ a graded $R
$-module and $C$ a graded classical 2-absorbing second submodule of $M$.
Let $I=\bigoplus_{\lambda \in G}I_{\lambda }$ be a graded ideal of $R$ and $%
\lambda \in G$ with $I_{\lambda }\nsubseteq Ann_{R}(C).$ Then $I_{\lambda }C$
is a graded classical 2-absorbing second submodule of $M.$
\end{theorem}
\begin{proof}
Since $I_{\lambda }\nsubseteq Ann_{R}(C),$ $I_{\lambda
}C\neq 0.$ Now, let $r_{\alpha },s_{\beta },t_{\gamma }\in h(R)$ and $U$ be a
completely graded irreducible submodule of $M$ such that $r_{\alpha
}s_{\beta }t_{\gamma }I_{\lambda }\ C\subseteq U.$ By Lemma 3.2,
we have either $r_{\alpha }s_{\beta }I_{\lambda }C\subseteq U$ or $\
t_{\gamma }I_{\lambda }C\subseteq U$ or $r_{\alpha }s_{\beta }t_{\gamma
}C\subseteq U.$ If $r_{\alpha }s_{\beta }I_{\lambda }C\subseteq U$ or $%
t_{\gamma }I_{\lambda }C\subseteq U,$ then we are done. If $r_{\alpha
}s_{\beta }t_{\gamma }C\subseteq U,$ then either $\ s_{\beta }t_{\gamma
}C\subseteq U$ or $r_{\alpha }t_{\gamma }C\subseteq U$ or $r_{\alpha
}s_{\beta }C\subseteq U$ as $C$ is a graded classical 2-absorbing second
submodule, it follows that either $s_{\beta }t_{\gamma }I_{\lambda
}C\subseteq U$ or $r_{\alpha }t_{\gamma }I_{\lambda }C\subseteq U$ or $%
r_{\alpha }s_{\beta }I_{\lambda }C\subseteq U.$ Thus $I_{\lambda }C$ is a
graded classical 2-absorbing second submodule of $M.$
\end{proof}
Let $R$ be a $G$-graded ring and $M$, $M^{\prime }$ graded $R$-modules. Let $f:M\rightarrow M^{\prime }$ be an $R$-module homomorphism. Then $f$  is
said to be a graded homomorphism if $f(M_{\alpha })\subseteq M_{\alpha }^{\prime }$
for all $\alpha $ $\in G$, (see \cite{28}.) The category of graded $R$-modules possesses direct sums, products injective and projective limits. A graded homomorphism that is an injective function will be referred to simply as a monomorphism and a graded homomorphism that is a surjective function will be called an epimorphism.

\begin{theorem}
Let $R$ be a $G$-graded ring and $M_{1}, M_{2}$ be two
graded $R$-modules. Let $f:M_{1}\rightarrow M_{2}$ be a graded monomorphism.
Then we have the following:
\begin{enumerate}[\upshape (i)]

\item If $C_{1}$ is a graded classical 2-absorbing second submodule of $%
M_{1}$, then $f(C_{1})$ is a graded classical 2-absorbing second submodule
of $f(M_{1}).$

\item If $C_{2}$ is a graded classical 2-absorbing second submodule of $%
f(M_{1})$, then $f^{-1}(C_{2})$ is a graded classical 2-absorbing second
submodule of $M_{1}.$\

\end{enumerate}
\end{theorem}
\begin{proof}
$(i)$ Assume that $C_{1}$\ is a graded classical 2-absorbing second submodule of $%
M_{1}.$ It is easy to see $f(C_{1})\neq 0.$ Now, Let $r_{\alpha },$ $s_{\beta },$ $t_{\gamma }\in h(R)$ and $U_{2}$
\ be a completely graded irreducible submodule of $f(M_{1})$ such that $%
r_{\alpha }s_{\beta }t_{\gamma }f(C_{1})\subseteq U_{2}.$ Consequently $%
r_{\alpha }s_{\beta }t_{\gamma }C_{1}\subseteq f^{-1}(U_{2}).$ By \cite[Lemma 2.11(ii)]{5}, we have $f^{-1}(U_{2})$ is a completely graded
irreducible submodule of $M_{1}$. Then either \ $r_{\alpha }s_{\beta
}C_{1}\subseteq f^{-1}(U_{2})$ or $\ s_{\beta }t_{\gamma }C_{1}\subseteq
f^{-1}(U_{2})$ or $r_{\alpha }s_{\beta }C_{1}\subseteq f^{-1}(U_{2})$ as $%
C_{1}$ is a graded classical 2-absorbing second submodule of $M_{1}.$ If $%
r_{\alpha }s_{\beta }C_{1}\subseteq f^{-1}(U_{2}),$ then $r_{\alpha
}s_{\beta }f(C_{1})=f(r_{\alpha }s_{\beta }C_{1})\subseteq
f(f^{-1}(U_{2}))=U_{2}\cap f(M_{1})=U_{2}.$ Similarly, if \ $s_{\beta
}t_{\gamma }C_{1}\subseteq f^{-1}(U_{2}),$ we get $s_{\beta }t_{\gamma
}f(C_{1})\subseteq U_{2}\ $, also if $r_{\alpha }s_{\beta }C_{1}\subseteq
f^{-1}(U_{2})$, we get $r_{\alpha }s_{\beta }f(C_{1})\subseteq U_{2}.$ Thus $%
f(C_{1})$ is a graded classical 2-absorbing second submodule of $f(M_{1}).$%

$(ii)$ Assume that $C_{2}$ is a graded classical 2-absorbing second
submodule of $f(M_{1}).$ It is easy to see $f^{-1}(C_{2})\neq 0.$ Now, let $%
r_{\alpha },$ $s_{\beta },$ $t_{\gamma }\in h(R)$ and $U_{1}$  be a
completely graded irreducible submodule of $M_{1}$ such that $r_{\alpha
}s_{\beta }t_{\gamma }f^{-1}(C_{2})\subseteq U_{1}.$ Consequently $r_{\alpha
}s_{\beta }t_{\gamma }C_{2}=r_{\alpha }s_{\beta }t_{\gamma }(f(M_{1})\cap
C_{2})=r_{\alpha }s_{\beta }t_{\gamma }ff^{-1}(C_{2})=f(r_{\alpha }s_{\beta
}t_{\gamma }f^{-1}(C_{2}))\subseteq f(U_{1}).$ By \cite[Lemma 2.11(i)]{5}, we have $f(U_{1})$ is a completely graded irreducible submodule of $%
f(M_{1}).$ Then either $r_{\alpha }s_{\beta }C_{2}\subseteq f(U_{1})$ or $%
s_{\beta }t_{\gamma }C_{2}\subseteq f(U_{1})$ or $r_{\alpha }t_{\gamma
}C_{2}\subseteq f(U_{1})$ as $C_{2}$\ is a graded classical 2-absorbing
second submodule of $f(M_{1}).$ This yields that either \ $r_{\alpha
}s_{\beta }f^{-1}(C_{2})\subseteq U_{1}\ $or $s_{\beta }t_{\gamma
}f^{-1}(C_{2})\subseteq U_{1}$ or $r_{\alpha }t_{\gamma
}f^{-1}(C_{2})\subseteq U_{1}.$ Thus $f^{-1}(C_{2})$ is a graded classical
2-absorbing second submodule of $M_{1}.$
\end{proof}


 \section{Graded strongly classical 2-absorbing second submodules}
\begin{definition}
Let $R$ be a $G$-graded ring and $M$ a graded $R$-module. A non-zero graded submodule $C$ of $M$ is said to be \textit{a graded
strongly classical 2-absorbing second submodule of $M$} if whenever $%
r_{\alpha }$, $s_{\beta },$ $t_{\gamma }\in h(R)$, $U_{1},$ $U_{2}$ and $%
U_{3}$ are completely graded irreducible submodules of $M,$ and $r_{\alpha
}s_{\beta }t_{\gamma }C\subseteq U_{1}\cap U_{2}\cap U_{3}$, then $r_{\alpha
}s_{\beta }C\subseteq U_{1}\cap U_{2}\cap U_{3}$ or $\ s_{\beta }t_{\gamma
}C\subseteq U_{1}\cap U_{2}\cap U_{3}$ or $r_{\alpha }t_{\gamma }C\subseteq
U_{1}\cap U_{2}\cap U_{3}.$ We say that $M$ is a graded strongly \textit{%
classical 2-absorbing second module} if $M$  is a graded strongly
classical 2-absorbing second submodule of itself.
\end{definition}


The following results give us a characterization of graded strongly classical 2-absorbing second submodule of a graded module.

\begin{theorem}
Let $R$ be a $G$-graded ring, $M$ a graded $R$-module and $C$ a non-zero graded submodule of $M$. Then the following
statements are equivalent:
\begin{enumerate}[\upshape (i)]

\item $C$ is a graded strongly classical 2-absorbing second submodule of $M$;

\item For every $r_{\alpha }$, $s_{\beta },$ $t_{\gamma }\in h(R)$ and
every graded submodule $N$ of $M$ with $r_{\alpha }s_{\beta }t_{\gamma
}C\subseteq N,$ then either $r_{\alpha }s_{\beta }C\subseteq N$ or $\
s_{\beta }t_{\gamma }C\subseteq N$ or $r_{\alpha }t_{\gamma }C\subseteq N;$

\item For every $r_{\alpha }$, $s_{\beta },$ $t_{\gamma }\in h(R),$ either
$r_{\alpha }s_{\beta }t_{\gamma }C=r_{\alpha }s_{\beta }C$ \ or $r_{\alpha
}s_{\beta }t_{\gamma }C=\ s_{\beta }t_{\gamma }C$ or $r_{\alpha }s_{\beta
}t_{\gamma }C=r_{\alpha }t_{\gamma }C.$
\end{enumerate}
\end{theorem}
\begin{proof}
$(i)\Rightarrow (ii)$ Assume that $C$ is a graded strongly
classical 2-absorbing second submodule of $M.$ Let $r_{\alpha }$, $s_{\beta
},$ $t_{\gamma }\in h(R)$ and $N$ \ be a graded submodule of $M$ such that $%
r_{\alpha }s_{\beta }t_{\gamma }C\subseteq N.$ Assume on the contrary that $%
r_{\alpha }s_{\beta }C\nsubseteq N$,$\ s_{\beta }t_{\gamma }C\nsubseteq N$
and $r_{\alpha }t_{\gamma }C\nsubseteq N.$ Then there exist completely
graded irreducible submodules $U_{1},$ $U_{2}$ and $U_{3}$ of $M$ such
that $N\subseteq U_{1},$ $N\subseteq U_{2}$ and $N\subseteq U_{3}$ but $%
r_{\alpha }s_{\beta }C\nsubseteq U_{1},$ $\ s_{\beta }t_{\gamma }C\nsubseteq
U_{2}$ and $r_{\alpha }t_{\gamma }C\nsubseteq U_{3}$ by \cite[Lemma 2.3]{5}. Hence $r_{\alpha }s_{\beta }t_{\gamma
}C\subseteq $ $U_{1}\cap U_{2}\cap U_{3}.$ Since $C$ is a graded strongly
classical 2-absorbing second submodule of $M,$ we get either $r_{\alpha
}s_{\beta }C\subseteq $ $U_{1}\cap U_{2}\cap U_{3}$ or $\ s_{\beta
}t_{\gamma }C\subseteq U_{1}\cap U_{2}\cap U_{3}$ or $r_{\alpha }t_{\gamma
}C\subseteq $ $U_{1}\cap U_{2}\cap U_{3}$ which are contradiction.

$(ii)\Rightarrow (iii)$ Assume $(ii)$ is hold. Let $r_{\alpha }$, $s_{\beta
},$ $t_{\gamma }\in h(R)$. Then  $ r_{\alpha }s_{\beta }t_{\gamma }C$ is a graded
submodule of $M$. Since $r_{\alpha }s_{\beta }t_{\gamma }C\subseteq
r_{\alpha }s_{\beta }t_{\gamma }C,$ by $(ii),$ we have either $r_{\alpha
}s_{\beta }C\subseteq r_{\alpha }s_{\beta }t_{\gamma }C$ or $\ s_{\beta
}t_{\gamma }C\subseteq r_{\alpha }s_{\beta }t_{\gamma }C$ or $r_{\alpha
}t_{\gamma }C\subseteq r_{\alpha }s_{\beta }t_{\gamma }C.$ This yields that
either $r_{\alpha }s_{\beta }t_{\gamma }C=r_{\alpha }s_{\beta }C$ or $%
r_{\alpha }s_{\beta }t_{\gamma }C=s_{\beta }t_{\gamma }C$ or $r_{\alpha
}s_{\beta }t_{\gamma }C=r_{\alpha }t_{\gamma }C\ .$

$(iii)\Rightarrow (i)$ Trivial.
\end{proof}

\begin{lemma}
Let $R$ be a $G$-graded ring, $M$ a graded $R$-module and
$C$ a graded strongly classical 2-absorbing second submodule of $M$. Let $%
I=\bigoplus_{\gamma \in G}I_{\gamma }$ be a graded ideal of $R$. Then for
every $r_{\alpha },s_{\beta }$ $\in h(R)$, $\gamma \in G$ and graded
submodule $N$ of $M$ with $r_{\alpha }s_{\beta }I_{\gamma }C\subseteq N$,
either $r_{\alpha }I_{\gamma }C\subseteq N$ or $\ s_{\beta }I_{\gamma
}C\subseteq N$ or $r_{\alpha }s_{\beta }C\subseteq N$.
\end{lemma}
\begin{proof}
By using Theorem 4.2 the proof is similar to the
poof of Lemma 3.2, so we omit it.
\end{proof}

\begin{lemma}
Let $R$ be a $G$-graded ring, $M$ a graded $R$-module
and $C$  a graded strongly classical 2-absorbing second submodule of $M$.
Let $I=\bigoplus_{\beta \in G}I_{\beta }$ and $J=\bigoplus_{\gamma \in
G}J_{\gamma }$\ be a graded ideals of $R.$ Then for every $r_{\alpha }$ $\in
h(R)$, $\beta ,\gamma \in G$ and graded submodule $N$ of $M$ \ with $%
r_{\alpha }I_{\beta }J_{\gamma }C\subseteq N$, either $\ r_{\alpha }I_{\beta
}C\subseteq N$ or $r_{\alpha }J_{\gamma }C\subseteq N$ or $\ I_{\beta
}J_{\gamma }C\subseteq N.$
\end{lemma}
\begin{proof}
By using Theorem 4.2 and Lemma 4.3, the proof is similar to the poof of Lemma 3.3, so we
omit it.
\end{proof}

\begin{theorem}
Let $R$ be a $G$-graded ring, $M$ a graded $R$-module and $C$ a non-zero graded submodule of $M$. Let $I=\bigoplus_{\alpha \in G}I_{\alpha },$ $J=\bigoplus_{\beta \in
G}J_{\beta }$ and $K=\bigoplus_{\gamma \in G}K_{\gamma }$ be a graded ideals
of $R$. Then the following statements are equivalent:
\begin{enumerate}[\upshape (i)]

\item $C$ is a graded strongly classical 2-absorbing second submodule of $%
M $.

\item For every $\alpha ,$ $\beta ,$ $\gamma \in G$ and graded submodule $N$
of $M$ with $I_{\alpha }J_{\beta }K_{\gamma }C\subseteq N$, either $%
J_{\beta }K_{\gamma }C\subseteq N$ or $I_{\alpha }J_{\beta }C\subseteq N$ or
$I_{\alpha }K_{\gamma }C\subseteq N.$
\end{enumerate}
\end{theorem}
\begin{proof}
By using Theorem 4.2 and Lemma 4.4, the proof
is similar to the proof of Theorem 3.4, so we omit it.
\end{proof}
\begin{theorem}
Let $R$ be a $G$-graded ring, $M$ a graded $R$%
-module and $C$ a non-zero graded submodule of $M$ and $N$ a graded
submodule of $M$ with\ $C\nsubseteq N.$ Then $C$ is a graded strongly
classical 2-absorbing second submodule of $M$ if and only if $(N:_{R}C)$ is
a graded 2-absorbing ideal of $R.$
\end{theorem}
\begin{proof}
Assume that $C$ is a graded strongly classical 2-absorbing second submodule
of $M.$ Since $C\nsubseteq N,$ $(N:_{R}C)\neq R.$ Now let $r_{\alpha },$ $s_{\beta
},$ $t_{\gamma }\in h(R)$ such that $r_{\alpha }s_{\beta }t_{\gamma }\in
(N:_{R}C).$ Then $r_{\alpha }s_{\beta }t_{\gamma }C\subseteq N.$ By
Theorem 4.2, we have either $r_{\alpha }s_{\beta }C\subseteq N$ or $\
s_{\beta }t_{\gamma }C\subseteq N$ or $r_{\alpha }t_{\gamma }C\subseteq N.$
Thus either $r_{\alpha }s_{\beta }\in (N:_{R}C)$ or $\ s_{\beta }t_{\gamma
}\in (N:_{R}C)$ or $r_{\alpha }t_{\gamma }\in (N:_{R}C).$ Therefore $%
(N:_{R}C)$ is a graded 2-absorbing ideal of $R.$ Conversely, Let $r_{\alpha
},$ $s_{\beta },$ $t_{\gamma }\in h(R)$ and $N$ be a graded submodule of $M$
with $r_{\alpha }s_{\beta }t_{\gamma }C\subseteq N.$ If $C\subseteq N,$ then
we are done.
Assume that $C\nsubseteq N.$ By our assumption, we have $(N:_{R}C)$ is a
graded 2-absorbing ideal of $R.$ Since $r_{\alpha }s_{\beta }t_{\gamma }\in
(N:_{R}C)$, we conclude that either $r_{\alpha }s_{\beta }C\subseteq N$ or $%
s_{\beta }t_{\gamma }C\subseteq N$ or $r_{\alpha }t_{\gamma }C\subseteq N$.
Hence $C$ is a graded strongly classical 2-absorbing second submodule of $M$ by Theorem 4.2.
\end{proof} 

Clearly every graded strongly 2-absorbing second submodule is a graded
strongly classical 2-absorbing second submodule. The following example shows that the
converse is not true in general.
\begin{example}
 Let $G=%
\mathbb{Z}
_{2},$ then $R=%
\mathbb{Z}
$ is a $G$-graded ring with $R_{0}=%
\mathbb{Z}
$ and $R_{1}=\{0\}.$ Let $M=%
\mathbb{Z}
_{6}\times
\mathbb{Q}
$ be a graded $R$-module with $M_{0}=%
\mathbb{Z}
_{6}\times
\mathbb{Q}
$ and $M_{1}=\{(0,0)\}.$ Then $M$ is a graded strongly classical 2-absorbing
second module which is not a graded strongly 2-absorbing second module.
\end{example}

\begin{theorem}
Let $R$ be a $G$-graded ring and $M_{1}, M_{2}$ be two graded $R$-modules.
Let $f:M_{1}\rightarrow M_{2}$ be a graded monomorphism. Then we have the
following.
\begin{enumerate}[\upshape (i)]

\item If $C_{1}$ is a graded strongly classical 2-absorbing second
submodule of $M_{1}$, then $f(C_{1})$ is a graded strongly classical
2-absorbing second submodule of $M_{2}$.

\item If $C_{2}$ is a graded strongly classical 2-absorbing second
submodule of $f(M_{1})$, then $f^{-1}(C_{2})$ is a graded strongly classical
2-absorbing second submodule of $M_{1}.$

\end{enumerate}
\end{theorem}
\begin{proof}
$(i)$ Assume that $C_{1}$ is a graded strongly classical 2-absorbing
second submodule of $M.$ It is easy to see $f(C_{1})\neq 0.$ Let $r_{\alpha
},s_{\beta },t_{\gamma }\in h(R).$ By Theorem 4.2, we have either $%
r_{\alpha }s_{\beta }t_{\gamma }C_{1}=r_{\alpha }s_{\beta }C_{1}$ or $%
r_{\alpha }s_{\beta }t_{\gamma }C_{1}=s_{\beta }t_{\gamma }C_{1}$ or $%
r_{\alpha }s_{\beta }t_{\gamma }C_{1}=r_{\alpha }t_{\gamma }C_{1}.$ We can
assume that $r_{\alpha }s_{\beta }t_{\gamma }C_{1}=r_{\alpha }s_{\beta
}C_{1}.$ Then $r_{\alpha }s_{\beta }t_{\gamma }f(C_{1})=f(r_{\alpha
}s_{\beta }t_{\gamma }C_{1})=f(r_{\alpha }s_{\beta }C_{1})=r_{\alpha }s_{\beta
}f(C_{1}).$ Hence $f(C_{1})$ is a graded strongly classical 2-absorbing
second submodule of $M_{2}$ by Theorem 4.2. \newline
$(ii)$ Use the technique of Theorem 3.9(ii), and apply Theorem 4.2.
\end{proof}
\begin{definition}
Let $S$ be a non-zero graded submodule of a graded $R$-module $M$. We say that $S$ is \textit{a graded weakly second submodule
of $M$ }if $r_{\alpha }s_{\beta }S\subseteq N,$ where $r_{\alpha }, s_{\beta
}\in h(R)$ and $N$ is a graded submodule of $M$, then either $r_{\alpha
}S\subseteq N$ or $\ s_{\beta }S\subseteq N.$
\end{definition}

\begin{lemma}
Let $R$ be a $G$-graded ring and $M$ be a graded $R$-module. If $S_{1}$ and $S_{2}$ are graded weakly second submodules of $M,$
then $S=S_{1}+S_{2}$ is a graded strongly classical 2-absorbing second submodule of $M.$
\end{lemma}
\begin{proof}
Assume that $S_{1}$ and $S_{2}$ are graded weakly second
submodules of $M$ and $S=S_{1}+S_{2}.$ Let $r_{\alpha },$ $s_{\beta },$ $%
t_{\gamma }\in h(R)$. As $r_{\alpha }s_{\beta }t_{\gamma }S_{1}\subseteq
r_{\alpha }s_{\beta }t_{\gamma }S_{1}$ and $S_{1}$ is a graded weakly second
submodule, then either $r_{\alpha }S_{1}\subseteq r_{\alpha }s_{\beta
}t_{\gamma }S_{1}$ or $\ s_{\beta }S_{1}\subseteq r_{\alpha }s_{\beta
}t_{\gamma }S_{1}$ or $t_{\gamma }S_{1}\subseteq r_{\alpha }s_{\beta
}t_{\gamma }S_{1}.$ This yields that either $r_{\alpha }S_{1}=r_{\alpha
}s_{\beta }t_{\gamma }S_{1}$ or $\ s_{\beta }S_{1}=r_{\alpha }s_{\beta
}t_{\gamma }S_{1}$ or $t_{\gamma }S_{1}=r_{\alpha }s_{\beta }t_{\gamma
}S_{1}.$ Similarly, we have either $r_{\alpha }S_{2}=r_{\alpha }s_{\beta
}t_{\gamma }S_{2}$ or $\ s_{\beta }S_{2}=r_{\alpha }s_{\beta }t_{\gamma
}S_{2}$ or $t_{\gamma }S_{2}=r_{\alpha }s_{\beta }t_{\gamma }S_{2}.$ We may
assume that $r_{\alpha }S_{1}=r_{\alpha }s_{\beta }t_{\gamma }S_{1}$.
Likewise assume that $s_{\beta }S_{2}=r_{\alpha }s_{\beta }t_{\gamma }S_{2}.$
Then $r_{\alpha }s_{\beta }t_{\gamma }S=r_{\alpha }s_{\beta }t_{\gamma
}S_{1}+r_{\alpha }s_{\beta }t_{\gamma }S_{2}=r_{\alpha }S_{1}+s_{\beta
}S_{2}=r_{\alpha }s_{\beta }S=r_{\alpha }s_{\beta }S_{1}+r_{\alpha }s_{\beta
}S_{2}=r_{\alpha }s_{\beta }S.$ By Theorem 4.2, we have $S$ is a
graded strongly classical 2-absorbing second submodules of $M$.
\end{proof}
Let $R_{i}$ be a graded commutative ring with identity and $M_{i}$ be a
graded $R_{i}$-module, for $i=1,2$. Let $R=R_{1}\times R_{2}.$ Then $%
M=M_{1}\times M_{2}$ is a graded $R$-module and each graded submodule $C$ of
$M$ is of the form $C=C_{1}\times C_{2}$ for some graded submodules $C_{1}$
of $M_{1}$ and $C_{2}$ of $M_{2}$ (see \cite{28}.)
\begin{theorem}
Let $R=R_{1}\times R_{2}$ be a $G$-graded ring and $M=M_{1}\times M_{2}$ a graded $R$-module where $M_{1}$ is a graded $R_{1}$%
-module and $M_{2}$ is a graded $R_{2}$-module. Let $C_{1}$ and $C_{2}$ be a
non-zero graded submodules of $M_{1}$ and $M_{2},$ respectively.
\begin{enumerate}[\upshape (i)]

\item $C=C_{1}\times 0$ is a graded strongly classical 2-absorbing
second submodule of $M$ if and only if $C_{1}$ is a graded strongly
classical 2-absorbing second submodule of $M_{1}$.

\item $C=0\times C_{2}$ is a graded strongly classical 2-absorbing second
submodule of $M$ if and only if $C_{2}$ is a graded strongly classical
2-absorbing second submodule of $M_{2}$.

\item $C=C_{1}\times C_{2}$ is a graded strongly classical 2-absorbing
submodule of $M$ if and only if $C_{1}$ and $C_{2}$ are graded weakly
second submodules of $M_{1}$ and $M_{2},$ respectively.

\end{enumerate}
\end{theorem}
\begin{proof}
 $(i)$ Assume that $C=C_{1}\times 0$ is a graded strongly
classical 2-absorbing second submodule of $M$. From our assumption, $C$ is a
non-zero graded submodule, so $C_{1}\neq 0.$ Set $M^{\prime }=M_{1}\times 0.$
One can see that $C=C_{1}\times 0$ is a graded strongly classical
2-absorbing second submodule of $M^{\prime }.$ Also, observe that $M^{\prime
}\simeq M_{1}$ and $C$ $\simeq C_{1}.$ Hence $C_{1}$ is a graded strongly
classical 2-absorbing second submodule of $M_{1}$ by Theorem 4.8. The converse is clear.

$(ii)$ It can be easily verified similar to $(i)$.

$(iii)$ Assume that $C=C_{1}\times C_{2}$ is a graded strongly classical
2-absorbing submodule of $M.$ We have to show that $C_{1}$ is a graded
weakly second submodule of $M_{1}.$ Since $C_{2}\neq 0,$ by \cite[Lemma 2.3]{5}, there exist a completely graded irreducible
submodule $U_{2}$ of $M_{2}$ such that $C_{2}\nsubseteq U_{2}.$ \ Let $%
r_{\alpha },$ $s_{\beta }\in h(R)$, $N$ be a graded submodule of $M_{1}$ and
$r_{\alpha }s_{\beta }C_{1}\subseteq N.$ Thus $(r_{\alpha },1)(s_{\beta
},1)(1,0)(C_{1}\times C_{2})=r_{\alpha }s_{\beta }C_{1}\times 0\subseteq
N\times U_{2}.$ As $C=$ $C_{1}\times C_{2}$ is a graded strongly classical
2-absorbing submodule of $M,$ by Theorem 4.2, we have either $%
(r_{\alpha },1)(s_{\beta },1)(C_{1}\times C_{2})=r_{\alpha }s_{\beta
}C_{1}\times C_{2}\subseteq N\times U_{2}$ or $(s_{\beta
},1)(1,0)(C_{1}\times C_{2})=s_{\beta }C_{1}\times 0\subseteq N\times U_{2}$
or $(r_{\alpha },1)(1,0)(C_{1}\times C_{2})=r_{\alpha }C_{1}\times
0\subseteq N\times U_{2}.$

If $r_{\alpha }s_{\beta }C_{1}\times C_{2}\subseteq N\times U_{2},$ then $%
C_{2}\subseteq U_{2}$ a contradiction. Which implies either $r_{\alpha
}C_{1}\subseteq N$ or $s_{\beta }C_{1}\subseteq N.$ Thus $C_{1}$ is a graded
weakly second submodule of $M_{1}.$ Similarly, we can show that $C_{2}$ is a
graded weakly second submodule of $M_{2}.$

Conversely, assume that $C_{1}$ and $C_{2}$ are graded weakly second
submodules of $M_{1}$ and $M_{2},$ respectively. Then it is clear that $%
C_{1}\times 0,$ $0\times C_{2}$ are graded weakly second submodules of $M.$
By Lemma 4.10, we have $C=C_{1}\times 0+0\times C_{2}$ is a graded
strongly classical 2-absorbing submodule of $M$.
\end{proof}

  \begin{center}{\textbf{Acknowledgments}}
  \end{center}
  The authors wish to thank sincerely the referees for their valuable comments and suggestions.


  {\textbf{  Conflict of Interest Statement:}}
    On behalf of all authors, the corresponding author states that there is no conflict of interest.

\bigskip\bigskip\bigskip\bigskip

\end{document}